\documentclass[a4paper,12pt]{article}
\usepackage{a4}
\usepackage{amsfonts}
\usepackage{amsmath}
\usepackage{theorem}
\usepackage[german,english]{babel}
\usepackage{hyperref}

\newcommand{\heute}{17 May 1999}

\newcommand{\Gro}[1]{Gr\"ob\-ner}
\newcommand{\HS}{\text{\sl HS}}
\newcommand{\Co}{\text{\sl Co}}

\newcommand{\f}[1][p]{\mathbb{F}_{#1}}
\newcommand{\fG}[1][p]{\f[#1]G}

\newcommand{\pres}[1]{\textsf{Present}}
\newcommand{\pathn}[1]{\texttt{pathnode}}
\newcommand{\Syl}[1]{\operatorname{Syl}_{#1}}
\newcommand{\McL}{\text{\sl McL}}
\newcommand{\Suz}{\text{\sl Suz}}
\newcommand{\Ru}{\text{\sl Ru}}
\newcommand{\LL}{\operatorname{LL}}
\newcommand{\RLL}{\operatorname{RLL}}
\newcommand{\eLL}{\operatorname{eLL}}
\newcommand{\eRLL}{\operatorname{eRLL}}
\newcommand{\Je}{\operatorname{Je}}
\newcommand{\Jen}{\operatorname{J}}
\newcommand{\lex}{\operatorname{lex}}
\newcommand{\MB}[1]{\textsf{MakeBasis}}
\newcommand{\mB}[1]{\texttt{makeBasis}}
\newcommand{\mJB}[1]{\texttt{makeJenningsBasis}}
\newcommand{\rPA}[1]{\texttt{regularPermutationAction}}
\newcommand{\gi}[1]{\texttt{groupInfo}}
\newcommand{\mam}[1]{\texttt{makeActionMatrices}}
\newcommand{\mnt}[1]{\texttt{makeNontips}}
\newcommand{\dms}[1]{\texttt{writeGroebnerBasis}}
\newcommand{\Mag}[1]{\textsf{Magma}}
\newcommand{\MA}[1]{\textsf{Meat\-Axe}}

\title{\Gro. bases for $p$-group algebras}
\author{David J. Green\footnote{Current address: Mathematical Institute,
Friedrich-Schiller-Universit\"at Jena, D-07737 Jena, Germany. Email:
\texttt{david.green@uni-jena.de}} \\ Department of Mathematics \\
University of Wuppertal \\ D--42097 Wuppertal, Germany. \\
\texttt{green@math.uni-wuppertal.de}}
\date{\heute}

\newtheorem{theorem}{Theorem}[section]
\newtheorem{proposition}[theorem]{Proposition}
\newtheorem{lemma}[theorem]{Lemma}
\theorembodyfont{\rmfamily}
\newtheorem{definition}[theorem]{Definition}
\newtheorem{example}[theorem]{Example}
\newtheorem{remark}[theorem]{Remark}

\newenvironment{proof}[1][]{\begin{trivlist} \item[\hskip\labelsep
\indent \emph{Proof#1.}]}{\foorp \end{trivlist}}
\newcommand{\foorp}{{\unskip\nobreak\hfil\penalty50
 \hskip1em\vadjust{}\nobreak\hfil \vrule height5pt width5pt depth0pt
 \parfillskip=0pt \finalhyphendemerits=0 \par}}

\begin{document}
\maketitle

\begin{abstract}
\noindent
Experiment shows that the reverse length-lexicographical word ordering
consistently yields far smaller \Gro. bases for modular $p$-group algebras
than the length-lexicographical ordering.
For the so-called Jennings word ordering, based on a special power-conjugate
group presentation, the associated monomial algebra is a group invariant.
The package \pres. finds \Gro. bases for these three orderings.
\end{abstract}

\subsection*{Note added 9 October 2009}
These notes document some experiments I did in the late 1990s to compare the
usefulness of local and well orderings for noncommutative \Gro. basis
calculations in modular group algebras of $p$-groups. They are not (or rather
no longer) intended for publication. They are provided in the hope that
the tables in section~\ref{section:results} may be of interest.

Part of the material here formed the basis for Chapter~1 of my book\footnote{%
D.~J. Green, \emph{\Gro. Bases and the Computation of Group Cohomology},
Lecture Notes in Math.\@ vol.~1828, Springer-Verlag, 2003, xii+138pp.}\@.
The system \pres. is no longer being maintained, but parts of it live
on in the $p$-group cohomology package\footnote{S.~A. King and D.~J. Green,
\emph{$p$-Group Cohomlogy Package} for the \textsf{Sage} computer algebra
system. \texttt{http://sage.math.washington.edu/home/SimonKing/Cohomology/}}
for \textsf{Sage}\@.

\section{Introduction}
Let~$G$ be a finite $p$-group.  One way to do computational homological
algebra over the modular group algebra~$\fG$ is to realise it as a finitely
presented associative algebra, together with a \Gro. basis for the relations
ideal.  There are then several methods available for computing minimal
projective resolutions~\cite{FGKK}, \cite{GSZ},~\cite{method}\@.

The algebra generators are required to lie in the Jacobson radical.
Usually one chooses generators~$g_i$ for the group, and then takes
the $a_i = g_i - 1$ as generators for~$\fG$.
See \cite{FFG}~or \cite{vorl} for an introduction to noncommutative
\Gro. bases.

The choice of word ordering is of prime importance in the design of
efficient \Gro. basis methods. This paper compares the usefulness in our
context of three word orderings, and announces a package \pres. for computing
\Gro. bases with respect to these orderings.  They are:
\begin{itemize}
\item The \textbf{length-lexicographical} ordering ($\LL$):
the standard ordering for homological algebra
over associative algebras.
\item The \textbf{reverse length-lexicographical} ordering ($\RLL$)
allows one to read off the radical layers.
Empirically, it always yields a small \Gro. basis.
\item The \textbf{Jennings} ordering (see Definition~\ref{defn:JenningsOrder})
is defined on special presentations, which are constructed using the
Jennings series and involve some redundant generators.
Broadly speaking, it is for $p$-group algebras what power-conjugate
presentations are for $p$-groups.
All \Gro. bases have the same size, and in fact the associated monomial algebra
is an invariant of the group's order.
\end{itemize}
The criteria for comparing orderings are:
\begin{itemize}
\item
Smallest size of \Gro. basis, taken over all choices of generators.
\item
How easy is it to realise a small \Gro. basis?  That is, how
sensitive is size of \Gro. basis to choice of generators?
\item
How many properties of the \Gro. basis (e.g.,~maximum length of a
reduced word) do not depend on choice of generators?
\end{itemize}

\noindent
In Sections \ref{section:LL}~and \ref{section:Jennings} we look at
the orderings one by one.  For each ordering, a method is presented
for computing the \Gro. basis for the relations ideal.
These methods were implemented in the package \pres.\@.
Data structures and other aspects of the implementation are discussed in
Section~\ref{section:implement}.  Results obtained using \pres.
are presented in Section~\ref{section:results}, where we draw conclusions
about the comparative usefulness of the orderings.

The package \pres. is written partly in \Mag. code and partly in C\@.
See the end of the paper for details of how to obtain it.

\paragraph{Acknowledgements}
This work was started while I was at the Institute for Experimental Mathematics
in Essen, Germany. I am very grateful to the Institute for the continuing use
of their computing facilities.

\section{\Gro. bases}
\label{section:LL}
First we recall the basics about \Gro. bases in free associative algebras.
Let~$M$ be the free associative monoid with~$1$ on a finite set~$A$.
That is, the elements of~$M$ are words in the alphabet~$A$.  Assume we have
an ordering on~$M$ which is admissible in the following sense (far weaker
than that
of~\cite{FGKK}):

\begin{definition}
\label{defn:admissible}
Let~$M$ be a free monoid.  An ordering on~$M$ is called admissible
if $u v_1 w \leq u v_2 w$ whenever $u,v_1,v_2,w$ are elements
of~$M$ with $v_1 \leq v_2$.
\end{definition}

\noindent
Now let~$A$ be a set of algebra generators for~$\fG$, all lying in the
Jacobson radical~$J(\fG)$.
Then $M$~is the set of monomials in the free associative
$\f$-algebra on~$A$.  Write $I$ for the relations ideal: the kernel
of projection from this free algebra to~$\fG$.
A word $w \in M$ is called a tip if it is tha largest word in the support
of some element of~$I$; if not, it is called a nontip.
That is, the nontips are the words whose images in~$\fG$ are linearly
independent of the images of their predecessors.  It follows immediately
that the images of the nontips are linearly independent in~$\fG$.

\begin{lemma}
\label{lemma:span}
With the above assumptions, the nontips form a basis for~$\fG$.
\end{lemma}

\begin{proof}
By Nakayama's Lemma, $J^N(\fG)$ is zero for some~$N$.  All
algebra generators lie in~$J$, and so only finitely many
words are nonzero in~$\fG$.  These span~$\fG$, and contain the
nontips as a spanning subset.
\end{proof}


\noindent
By admissibility, the tips form an ideal in the free monoid~$M$.
This ideal has a unique smallest generating set.  It consists of
the \emph{minimal tips}, a tip being called minimal if and only if all proper
subwords are nontips.
Define
\[ \mathcal{G} := \{ w - \nu(w) \mid \text{$w$ a minimal tip} \} \, , \]
where $\nu(w)$ is the linear combination of nontips which in~$\fG$ equals~$w$.
Then $\mathcal{G}$~and the length~$N$ words together generate~$I$.
So, in the presence of the constraint $A^N \subseteq I$, the
set~$\mathcal{G}$ is a \Gro. basis for~$I$: the unique completely reduced
\Gro. basis for this ordering.

\begin{remark}
For the length-lexicographical and Jennings orderings, the constraint
$A^N \subseteq I$ is not necessary.
\end{remark}

\subsection{The length-lexicographical ordering}
Denote by $\ell(w)$ the length of a word $w \in M$.  Putting an ordering
on the set~$A$ of algebra generators induces a lexicographical ordering
$\leq_{\lex}$~on $M$.  The length-lexicographical ordering $\leq_{\LL}$~on
$M$ is then defined as follows:
\[
\begin{tabular}{rl}
$w_1 \leq_{\LL} w_2$ if & $\ell(w_1) < \ell(w_2)$, \\
or & same length, and $w_1 \leq_{\lex} w_2$.
\end{tabular}
\]
This ordering is admissible.  Moreover, each word $w \in M$ has only
finitely many predecessors under~$\leq_{\LL}$.
We do not use the Buchberger algorithm to compute the \Gro. basis, because
for $p$-group algebras it takes an unreasonable amount of time to stop.
Rather, we exploit the fact that $\fG$ is finite-dimensional, which means
that it is feasible to list all the nontips.

\paragraph{Computing a \Gro. basis}
Construct~$G$ in \Mag. using a faithful permutation
representation.  Choose minimal generators $g_1,\ldots,g_r$
for~$G$,
and set $a_i = g_i-1$.  Take the~$a_i$
as the set~$A$ of algebra generators.  Record the matrix for the right
multiplication action of each~$a_i$, with basis the group
elements.

Using these matrices, each word is constructed as a linear combination of
the group elements, starting with the smallest word and proceeding in
$\LL$-order.  Gaussian elimination allows us to distinguish the tips from
the nontips.  We stop when we
have found all the nontips: their number is known in advance.
The list of minimal tips is deduced from the list of nontips by word
manipulation.  Change of basis gives us the matrix for the multiplication
action of each generator with respect to the basis of nontips,
and we can read off~$\nu(w)$ for each minimal tip~$w$ from these matrices.

\subsection{The reverse length-lexicographical ordering}
Again defined on the free monoid~$M$ on an ordered set~$A$, this
admissible ordering is the opposite of~$\LL$.  Namely,
\[ w_1 \leq_{\RLL} w_2 \quad \text{if and only if} \quad w_1 \geq_{\LL} w_2
\, . \]
Note that~$\RLL$ is not a well-ordering.

\begin{lemma}
The $\RLL$-nontips of length at least~$r$ constitute a basis for~$J^r(\fG)$.
Hence the length~$r$ $\RLL$-nontips are a basis for a complement
of $J^{r+1}$~in $J^r$.
\end{lemma}

\begin{proof}
The words of length at least~$r$ span $J^r$.
If a linear combination of length~$r$ words lies in~$J^{r+1}$,
then the largest length~$r$ word involved is a tip.
\end{proof}

\noindent
\Mag. can compute the smallest~$N$ such that $J^N$ is zero (using the
Jennings series, the dimensions of all radical layers can be computed).
So we could obtain the nontips by running through the words of length
at most~$N$ in $\RLL$ order.  But this could take forever. For example,
if $G$~is a Sylow $2$-subgroup of the sporadic finite simple group~$\Co_3$,
then $N=23$ and $A$~has size~$4$.  Hence the number of words of
length~$N$ is~$2^{46}$.

\paragraph{Computing a \Gro. basis}
A length~$r$ word is a tip if and only if it lies in the space spanned by its
length~$r$ $\RLL$-predecessors and by $J^{r+1}(\fG)$.
Using Gaussian elimination and the matrices for the algebra generators,
we can calculate a basis for~$J^{r+1}$ from one for~$J^r$.  So
we start with $r=0$ and work through to $r=N$.  All length~$r$ nontips
are words of the form $w.a$ with $a \in A$ and $w$~a length $r-1$ nontip.
We work through these words in $\RLL$-order.
Once we have the nontips, we proceed as for~$\LL$.

\section{The Jennings ordering}
\label{section:Jennings}
For $r \geq 1$, define $F_r(G)$ to be the $r$th dimension subgroup
of the finite $p$-group~$G$.  That is,
\[ F_r(G) = \{ g \in G \mid g-1 \in J^r(\fG) \} \, . \]
Clearly $G = F_1(G) \geq F_2(G) \geq \cdots$.  The~$F_r$
do eventually reach the trivial group, and they form a central series, the
Jennings series for~$G$.  This series need not be strictly decreasing, however.
These and more facts about the Jennings series are demonstrated
in Section~3.14 of~\cite{Benson:I}

\begin{definition}
Let~$G$ be a $p$-group of order~$p^n$.  Elements $g_1,\ldots,g_n$~of $G$ are
Jennings pc-generators for~$G$ if the first~$m_1$ elements~$g_i$ are minimal
generators for~$F_1 (G) / F_2 (G)$, the next~$m_2$ are minimal generators for
$F_2 / F_3$, and so on.
\end{definition}

\noindent
Jennings pc-generators do yield a polycyclic presentation for~$G$,
as we now see.

\begin{lemma}
Let $g_1, \ldots, g_n$ be Jennings pc-generators for~$G$.
For each $1 \leq i \leq n$, we have
$g_i^p \in \langle g_{i+1}, \ldots, g_n \rangle$;
and $[g_i, g_j] \in \langle g_{j+1} \ldots, g_n \rangle$
for $1 \leq i < j \leq n$.
\end{lemma}

\begin{proof}
It is known that $[F_r, F_s] \leq F_{r+s}$
and that $p$th powers of elements of~$F_r$ lie in~$F_{pr}$.
\end{proof}

\begin{example}
\label{ex:32G}
Let~$G = \langle a,b,c,\phi \rangle$
be the following semidirect product group of order~$32$: the subgroup
$\langle a,b,c \rangle$ is normal, and elementary abelian of order eight.
The order of~$\phi$ is four, and the effect of conjugation on the left
by~$\phi$ is
\[ a \longmapsto b \longmapsto c \longmapsto abc \, . \]
Then $F_2 = \langle ab, ac, \phi^2 \rangle$ and $F_3 = \langle ac \rangle$
are elementary abelian, and so
$a, \phi, ab, \phi^2, ac$ are Jennings pc-generators for~$G$.

The last three generators are Jennings pc-generators for~$F_2$.
Another family of Jennings pc-generators for~$F_2$ is
$ac, ab, \phi^2$.  This family cannot be extended to Jennings pc-generators
for~$G$,  since $[\phi, ab] = ac$.  Hence to find Jennings pc-generators
for~$G$, it is not enough to take minimal generators for~$G$ and add
on Jennings pc-generators for~$F_2$.
\end{example}

\begin{definition}
Let $g_1,\ldots,g_n$ be Jennings pc-generators for~$G$.  Then
$a_1, \ldots, a_n$ are Jennings generators for the group algebra~$\fG$,
where $a_i = g_i - 1$.  We say~$a_i$ has dimension~$r$ if
$a_i \in J^r(\fG) \setminus J^{r+1}(\fG)$, or equivalently
$g_i \in F_r \setminus F_{r+1}$.
\end{definition}

\noindent
Order the Jennings generators as follows: $a_1 < a_2 < \cdots < a_n$.
For a word~$w$ in these generators, define its dimension $\dim(w)$ in
the obvious additive way:
\[ \dim(a_{i_1} \ldots a_{i_r}) = \sum_{j=1}^r \dim(a_{i_j}) \, . \]

\begin{definition}
\label{defn:JenningsOrder}
The Jennings ordering is the following ordering on the monoid of
words in the Jennings generators:
\[ \begin{tabular}{rl}
$w_1 \leq_{\Jen} w_2$ if & $\dim(w_1) > \dim(w_2)$, \\
or & dimensions equal and $\ell(w_1) < \ell(w_2)$, \\
or & same dimension and length, and $w_1 \geq_{\lex} w_2$.
\end{tabular} \]
\end{definition}

\noindent Observe that the Jennings ordering is admissible; and that
$\dim(w) \geq \ell(w)$ for each word~$w$.
The next result tells us that it is straightforward to calculate
the \Gro. basis for the Jennings ordering.

\begin{proposition}
\label{prop:Jennings}
Let $a_1,\ldots, a_n$ be Jennings generators for~$\fG$.
The minimal tips are the words~$a_i^p$ and $a_j a_k$ with $j < k$.
The nontips are the words
$a_n^{e_n} a_{n-1}^{e_{n-1}} \ldots a_1^{e_1}$ with $0 \leq e_i \leq p-1$.
A nontip~$w$ lies in $J^r \setminus J^{r+1}$ if and only $\dim(w) = r$.
\end{proposition}

\begin{proof}
We just have to show that $a_i^p$ and $a_j a_k$ ($j<k$) are tips.
For then all nontips are of the form claimed, and the number of
nontips equals the number of claimed nontips.  So the nontips and
hence the minimal tips are as claimed.  The distribution in radical
layers then follows from Jennings' theorem
(Theorem 3.14.6 in~\cite{Benson:I})\@.

If $\dim(a_i)=r$ then $a_i^p$ lies in $\f F_{pr}$,
which is generated by the $a_j$ with $\dim(a_j) \geq pr$.
Hence~$a_i^p$ is a linear combination of words in these~$a_j$.
Each such word either has dimension greater than~$pr$ or has length one.
So $a_i^p$~is a tip, since it has length~$p$.

Similarly, if $\dim(a_j)=s$ and $\dim(a_k)=t$, then $c = [g_j, g_k]$
lies in~$F_{s+t}$.  Set $\gamma = c-1$, so that $\gamma \in \f F_{s+t}$.
Arguing as above, $\gamma$~is a linear combination of words smaller than
$a_j a_k$~or $a_k a_j$.  Since $g_j g_k = c g_k g_j$, we have
\begin{equation*}
a_j a_k = a_k a_j + \gamma + \gamma a_k + \gamma a_j + \gamma a_k a_j \, ,
\end{equation*}
and the right hand side is a linear combination of words
smaller than~$a_j a_k$.
\end{proof}

\begin{example}
Let~$G$ be the cyclic group of order four.  Then Jennings pc-generators for~$G$
are $g_1,g_2$ with $g_1^2 = g_2$, $g_2^2 = 1$ and $[g_1,g_2] = 1$.
The corresponding Jennings generators for~$\fG$ are $a_1 = g_1-1$ and
$a_2=g_2-1$.  The minimal tips are $a_1^2$, $a_1 a_2$ and $a_2^2$; the
nontips are $a_2 a_1, a_2, a_1, 1$ in ascending order; and the \Gro. basis is
\[
a_1^2 + a_2 \, , \qquad a_1 a_2 + a_2 a_1 \, , \qquad a_2^2 \, .
\]
\end{example}

\begin{remark}
Jennings pc-generators are usually not minimal group generators.  This means
that there are relations involving length one words.  However, the Jennings
ordering does guarantee that all tips have length at least two.
\end{remark}

\paragraph{Computing a \Gro. basis}
Get \Mag. to compute the Jennings series.  Use this to pick Jennings
pc-generators for~$G$.  The nontips and the minimal tips are known
by Proposition~\ref{prop:Jennings}.  Proceed as for~$\LL$.

\section{The implementation}
\label{section:implement}
In this section we provide an overview of the package \pres.,
and describe the data structures used.
Some components of \pres. are written in \Mag. code, others are written
in~C\@.  The C components use M.~Ringe's C \MA. to handle vectors and
matrices over finite fields.

Groups must be constructed in \Mag. as permutation groups.
The function \rPA. is provided to convert pc-groups (and hence matrix groups
too) into permutation groups using the regular permutation action.

\subsection{Selecting minimal generators}
\label{subsection:gsm}
To choose minimal generators for a $p$-group~$G$ in pc-presentation,
\pres. first asks \Mag. for the Frattini subgroup~$\Phi(G)$, and sets
$H = \Phi(G)$.  Then it selects an element $g \in G - H$, adds this to the
list of generators, and replaces $H$~by $\langle g, H\rangle$.  This
is repeated until $H$~is equal to~$G$.  At this point the elements~$g$
constitute a minimal generating set.

There remains the question of which element of~$G-H$ to select at each stage.
Two generator selection methods were investigated:
\begin{itemize}
\item
The most obvious method: pick $g \in G - H$ completely at random.
Minimal generators constructed in this way are called
\emph{arbitrary} minimal generators.
\item
Calculate the exponent of each element of $G-H$,
and pick~$g$ at random from amongst those with the smallest exponent.
Such minimal generators are called \emph{smallest exponent} minimal generators.
\end{itemize}
The merits of these two methods are compared in Section~\ref{section:results}\@.
By default, \pres. uses the smallest exponent method for~$\LL$, and
the arbitrary method for~$\RLL$\@.

\subsection{The nontips tree}
The group generators are passed to the C programs as a list of permutations.
For the Jennings ordering, generator dimensions are passed too.  The \Gro.
basis is now determined by the method for the chosen ordering given above.

The nontips are words, and all proper subwords of a nontip are again nontips.
Hence the most natural way to store the nontips is as a tree, in which
the children of a word~$w$ are the nontips of the form $w.a$ with $a \in A$,
the set of algebra generators.
On the other hand, the nontips constitute an ordered basis, and so we also want
to be able to address them as an array.

In \pres., the nontips are stored as a length~$|G|$ array of \pathn.s.
The class \pathn. contains
\begin{itemize}
\item the word~$w$ itself, its array index and its length
\item an array of pointers to its children, indexed by the elements of~$A$\@.
(If $w.a$ is a tip, then the corresponding pointer is null.)
\item a pointer to the parent word.  Also, which child of the parent it is.
\end{itemize}
The nontips are stored in the array in ascending order for~$\LL$, and in
descending order for $\RLL$ and the Jennings ordering.  This means that
the root of the tree is always located in the first entry of the array;
and that
the nontips tree can be built \emph{while} the nontips are being determined.

For the second point, observe that the $\LL$ nontips are determined in
ascending order.  Also, the number of length~$r$ $\RLL$ nontips is known
in advance, and they are determined in ascending order after all nontips
of smaller length have been determined.  Hence for $\LL$ and $\RLL$, the index
of each nontip is known the moment it is identified as a nontip.

Strictly, the Jennings nontips are an exception here.  They are written down
in lexicographical order, and then sorted into Jennings order.

\subsection{The components of \pres.}
The package \pres. consists of the package \MB. of \Mag. functions,
and four C programs.

The \Mag. function \mB. constructs nontips and action matrices for the
$\LL$ and $\RLL$ orderings.  A variant has a user-defined number of
attempts at finding the smallest \Gro. basis.
The function \mJB. constructs nontips and action matrices for the Jennings
ordering.  There is no need to have repeated attempts.

For each ordering, a flag allows the user to insist that the defining group
generators are used.  This eliminates the random component.

The C programs \mnt. and \mam. are invoked by the \Mag. functions to determine
the nontips and the action matrices respectively.  The program \dms.
can then be used to write out the \Gro. basis and the minimal tips.
The utility program \gi. prints out relevant statistics on the groups
by decoding the nontips file header.

\section{Results and conclusions}
\label{section:results}
The package \pres. was used to compare the orderings $\LL$~and $\RLL$.
The groups used in the comparison were: all $51$ groups of order~$32$;
all $267$ groups of order~$64$; and eight Sylow $p$-subgroups of
sporadic finite simple groups.

Define $\LL(G)$ to be the smallest size of a \Gro. basis for the relations
ideal with respect to the $\LL$ ordering.  Define $\RLL(G)$
similarly for~$\RLL$\@.  Using \pres. we can obtain empirical
approximations $\eLL(G)$~and $\eRLL(G)$ to these numbers.

By Propostion~\ref{prop:Jennings}, all \Gro. bases for the Jennings ordering
have the same size~$\Je(G)$.  This is $\frac12 n(n+1)$ for a group
of order~$p^n$.

\subsection{Groups of order 32}
There are 51 groups of order 32.  The smallest \Gro. basis found in twenty
attempts was recorded for each combination of: group, ordering ($\LL$~or $\RLL$)
and generator selection method (arbitrary or smallest).

For the~$\RLL$ ordering, the generator selection method made no difference.
For~$\LL$, smallest minimal generators yielded a smaller \Gro. basis in three
cases: one less element for two groups, and three less for one group.

The comparative performance of the two orderings is shown below.  The most
extreme difference was for the group with Hall--Senior number~$43$, where
$\eLL = 31$ and $\eRLL = 10$.
\[
\renewcommand{\arraystretch}{1.2}
\begin{array}{@{}l||c|c|c|c|c|c@{}}
d := \eLL - \eRLL & d < 0 & d = 0 & 1 \leq d \leq 3 & 4 \leq d \leq 6 &
7 \leq d \leq 9 & d \geq 10 \\
\hline
\text{No.\@ of groups} & 0 & 17 & 16 & 8 & 7 & 3
\end{array}
\]

\subsection{Groups of order 64}
\label{subsection:64}
There are 267 groups of order 64.  The smallest \Gro. basis found in twenty
attempts was recorded for each combination of group, ordering
and generator selection method.

For the~$\RLL$ ordering, smallest minimal generators yielded a larger \Gro.
basis in $63$ cases, and a smaller \Gro. basis in one case.  Each time, the
difference was only one element.
For~$\LL$, smallest minimal generators yielded a smaller \Gro. basis in $97$
cases, with mean difference~$3.3$; and a larger \Gro. basis in $69$ cases,
with mean difference~$4.6$.

The comparative performance of the two orderings is shown below.  The most
extreme difference was for the group with Hall--Senior number~$187$, where
$\eLL = 57$ and $\eRLL = 10$.
\[
\renewcommand{\arraystretch}{1.2}
\begin{array}{@{}l||c|c|c|c|c|c@{}}
d := \eLL - \eRLL &
d < 0 &
d = 0 &
1 \leq d \leq 8 &
9 \leq d \leq 16 &
17 \leq d \leq 32 &
d \geq 33 \\
\hline
\text{No.\@ of groups} & 0 & 29 & 112 & 73 & 50 & 3
\end{array}
\]


\subsection{Sylow $p$-subgroups of sporadic finite simple groups}
One of the main applications driving the development of group cohomology
software is calculating the cohomology rings of sporadic finite
simple groups.  So Sylow $p$-subgroups of these groups are
important examples for \pres..  Moreover, they are a good source of larger
$p$-groups.

The sensitivity of \Gro. basis size to generator choice was investigated
by looking at four such Sylow $p$-subgroups.
The program was run thirty times for each combination of: group, ordering
($\LL$~or $\RLL$) and generator selection method.
The results are shown in Tables \ref{table:LLsmall}~and \ref{table:RLLsmall}\@.
We give the smallest and largest sizes of \Gro. basis found, together with
mean and standard deviation.

\begin{table}
\makebox[\linewidth][c]{$
\begin{array}{l|l||r|r|r|r||r|r|r|r}
\multicolumn2{c||}{} & \multicolumn4{c||}{\text{Arbitrary}} &
\multicolumn4{c}{\text{Smallest exponent}} \\
\hline
\text{Group} & \text{Order} & \text{Min} & \text{Max} &
\multicolumn1{c|}{\mu} & \multicolumn1{c||}{\sigma}
&
\text{Min} & \text{Max} &
\multicolumn1{c|}{\mu} & \multicolumn1{c}{\sigma}
\\
\hline
\Syl2(\HS) & 2^9 & 128 & 338 & 261 & 50 &
104 & 127 & 114 & 5 \\ 
\Syl2(M_{24}) & 2^{10} & 433 & 906 & 706 & 101 &
183 & 543 & 373 & 81 \\ 
\Syl2(\Co_3) & 2^{10} & 502 & 871 & 690 & 92 &
212 & 574 & 405 & 85 \\ 
\Syl3(\McL) & 3^6 & 126 & 333 & 258 & 50 &
133 & 245 & 197 & 43 \\ 
\hline
\end{array}
$}
\caption{Generator selection methods compared for~$\LL$\@.
Each sample size: 30.}
\label{table:LLsmall}
\end{table}

\begin{table}
\makebox[\linewidth][c]{$
\begin{array}{l|l||r|r|r|r||r|r|r|r}
\multicolumn2{c||}{} & \multicolumn4{c||}{\text{Arbitrary}} &
\multicolumn4{c}{\text{Smallest exponent}} \\
\hline
\text{Group} & \text{Order} & \text{Min} & \text{Max} &
\multicolumn1{c|}{\mu} & \multicolumn1{c||}{\sigma}
&
\text{Min} & \text{Max} &
\multicolumn1{c|}{\mu} & \multicolumn1{c}{\sigma}
\\
\hline
\Syl2(\HS) & 2^9 & 10 & 13 & 11.0 & 1.1 &
11 & 13 & 11.9 & 0.9 \\
\Syl2(M_{24}) & 2^{10} & 15 & 22 & 17.3 & 1.5 &
15 & 20 & 17.1 & 1.0 \\
\Syl2(\Co_3) & 2^{10} & 13 & 21 & 15.5 & 1.7 &
13 & 21 & 16.2 & 2.0 \\
\Syl3(\McL) & 3^6 & 11 & 14 & 11.8 & 1.0 &
12 & 14 & 12.6 & 0.9 \\
\hline
\end{array}
$}
\caption{Generator selection methods compared for~$\RLL$\@.
Each sample size: 30.}
\label{table:RLLsmall}
\end{table}

In Table~\ref{table:allspor}, we compare smallest \Gro. basis sizes for
all three orderings by looking at eight Sylow $p$-subgroups of
sporadic finite simple groups.  Each empirical value is based on at least
twenty calculations.

\begin{table}
\renewcommand{\arraystretch}{1.2}
\makebox[\linewidth][c]{$
\begin{array}{l||r|r|r|r|r|r||r|r}
& \multicolumn6{c||}{\text{Sylow $2$-subgroup of}} &
\multicolumn2{c@{}}{\text{$\Syl3$ of}} \\
& M_{22} & \HS & M_{24} & \Co_3 & \Suz & \Ru & \McL & \Suz \\
\hline \hline
\text{Order} & 2^7 & 2^9 & 2^{10} & 2^{10} & 2^{13} & 2^{14} & 3^6 & 3^7 \\
\hline
\eLL & 34 & 104 & 150 & 236 & 2669 & 3111 & 126 & 417 \\
\hline
\eRLL & 8 & 9 & 15 & 13 & 21 & 18 & 11 & 13 \\
\hline
\Je & 28 & 45 & 55 & 55 & 91 & 105 & 21 & 28
\end{array}
$}
\caption{Smallest known \Gro. bases.  Each sample size at least 20.}
\label{table:allspor}
\end{table}

\subsection{Conclusions}
\Gro. bases for the length-lexicographical ordering~$\LL$ are consistently
much larger than for the reverse length-lexicographical ordering~$\RLL$.
In Table~\ref{table:allspor}, the size of the smallest $\LL$ \Gro. basis
lies between 14\% and 33\% of the group order.  This means that
the $\LL$ ordering is unsuitable for $p$-group algebras.

By contrast, $\RLL$ consistently yields very small \Gro. bases:
the size behaving very roughly as the logarithm of the group order.  Moreover,
$\RLL$~allows one to read off the radical layers in the group algebra.  Hence
the $\RLL$ ordering is well-suited for computing
with $p$-group algebras.

The Jennings ordering also seems to be suitable for computing with $p$-group
algebras.  Finding Jennings generators is easy, and it would appear that all
Jennings generating sets are equally useful.  Many properties of a \Gro.
basis are invariants of the group's order, including the size.  \Gro. bases
are not as small as for~$\RLL$, but this is compensated for by the fact that
fewer multiplication operations are necessary to obtain the action of
the average minimal tip.

The smallest exponent generator selection method seems to deliver a
small $\LL$ \Gro. basis more often than the arbitrary method does:
see Table~\ref{table:LLsmall}.  However, there are numerous exceptions
amongst the groups of order~$64$.
For the $\RLL$ ordering, all evidence suggests that the smallest exponent
method should be avoided: it is no more reliable at yielding a small \Gro.
basis than the arbitrary method, and sometimes misses the
smallest \Gro. bases.

\paragraph{Software availability}
The package \pres. is available at
{\selectlanguage{german}%
http://www.""math.""uni-wuppertal.""de/\~{}green/""software.""html},
as is a copy of the C \MA..  The most recent version of the C \MA.
is available at
{\selectlanguage{german}%
http://www-gap.""dcs.""st-and.""ac.""uk/\~{}gap}
as a GAP 3 share
package.  The \Mag. home page is at
{\selectlanguage{german}%
http://www.""maths.""usyd.""edu.""au:8000/""u/""magma/}.

\end{document}